\documentclass[a4paper,11pt]{amsart}
\usepackage{amsmath,amssymb,amsthm,amscd,xypic}

\title[Homotopy category of projectives and Brown representability]
{The dual of the homotopy category of projective modules satisfies Brown representability}

\author{George Ciprian Modoi}
\address{Babe\c s--Bolyai University, Faculty of Mathematics and Computer Science \\
1, Mihail Kog\u alniceanu, 400084 Cluj--Napoca, Romania}
\email{cmodoi@math.ubbcluj.ro}

\thanks{Research supported by CNCS-UEFISCDI grant PN-II-RU-TE-2011-3-0065}


\date{\today}

\newcommand{\la}{\longrightarrow}

\newcommand{\N}{\mathbb{N}}
\newcommand{\Z}{\mathbb{Z}}
\newcommand{\Q}{\mathbb{Q}}

\DeclareMathOperator{\Hom}{Hom}

\DeclareMathOperator{\holim}{\underleftarrow{\textrm{holim}}}

\DeclareMathOperator{\Ho}{H}

\newcommand{\A}{\mathcal{A}}

\newcommand{\C}{\mathcal{C}}

\newcommand{\CS}{\mathcal{S}}
\newcommand{\T}{\mathcal{T}}

\newcommand{\U}{\mathcal{U}}

\newcommand{\G}{\mathcal{G}}

\newcommand{\ModR}{\hbox{\rm Mod-}R}

\newcommand{\FlatR}{\hbox{\rm Flat-}R}

\newcommand{\ProjR}{\hbox{\rm Proj-}R}

\newcommand{\Ab}{\mathcal{A}b}

\newcommand{\opp}{^\textit{o}}

\newcommand{\Prod}[1]{\mathrm{Prod}({#1})}
\newcommand{\Prd}{\mathrm{Prod}}

\newcommand{\Htp}[1]{\mathbf{K}({#1})}

\theoremstyle{plain}
\newtheorem{thm}{Theorem}
\newtheorem{lem}[thm]{Lemma}

\newtheorem{cor}[thm]{Corollary}

\theoremstyle{definition}

\theoremstyle{remark}
\newtheorem{rem}[thm]{Remark}

\newenvironment{pf}{\noindent {\it Proof of}}

\begin{document}

\begin{abstract}
We show that the dual of the homotopy category of projective modules over an arbitrary ring satisfies 
Brown representability. 
\end{abstract}

\maketitle 


\section{Introduction and the main result}

This short note belongs to a series of papers which deal with Brown representability.
In \cite{MP} we gave a new proof of the fact that a well--generated triangulated category 
satisfies Brown representability, by using that every object is a homotopy colimit of 
a suitable chosen directed tower of objects constructed starting with the generators. 
Next we adapted in \cite{MD} this method in the sense 
made precise in Lemma \ref{decon} bellow, and we found a formal criterion for 
the dual of Brown representability in a triangulated category 
with products. This formal result was used first in the same paper \cite{MD},
for characterizing when the homotopy category of complexes of all modules satisfy the dual 
of Brown representability, and second in \cite{MBD} in order to show that the derived category 
of a Grothendieck category, 
satisfying certain additional hypothesis, satisfies the dual of Brown representability. In 
the present work we use the same instrument for proving that the dual of the homotopy category of 
projective module over an arbitrary ring satisfies Brown representability, confirming once again 
the usefulness of our formal result.  Note that the homotopy category of projectives is a key 
ingredient of the new point of view over Grothendieck duality given by Neeman in \cite{NKF}. 
Next in \cite{NKP},
the same author constructed a set of cogenerators in of this homotopy category, which was shown to 
be $\aleph_1$--compactly generated but, in general, not compactly generated.

Let $R$ be a ring (associative with one). In the sequel we shall work with the category of 
(complexes up to homotopy of) right $R$--modules. Thus the word ``module'' means ``right module''
and whenever we have to deal with left modules we state it explicitly. We denote by $\ModR$ 
the category of all modules, and we consider the full subcategories $\FlatR$ and $\ProjR$ 
consisting of flat, respectively projective modules. Complexes (of modules) are 
cohomologically graded, that is a complex is a sequence of the form
\[X=\cdots\to X^{n-1}\stackrel{d^{n-1}}\la X^n\stackrel{d^n}\la X^{n+1}\to\cdots\]
with $X^n\in\ModR$, $n\in\Z$, and $d^nd^{n-1}=0$. Morphisms of complexes are collections of linear 
maps commuting with differentials. Two maps of complexes $(f^n)_{n\in\Z},(g^n)_{n\in\Z}:X\to Y$ are 
homotopically equivalent if there are $s^n:X^n\to Y^{n-1}$, for all
$n\in\Z$, such that $f^n-g^n=d_Y^{n-1}s^n+s^{n+1}d_X^n$.
The homotopy category $\Htp\ModR$ has as objects 
all complexes and as morphisms equivalence classes of morphisms of complexes up to homotopy. 
It is well--known 
that $\Htp\ModR$ is a triangulated category with (co)products. Its suspension functor, 
denoted by $[1]$, is defined as follows: $X[1]^n=X^{n+1}$ 
and $d_{X[1]}^n=-d_X^{n+1}$. Let denote $\Htp\FlatR$ and 
$\Htp\ProjR$ the full subcategories of $\Htp\ModR$ consisting of those complexes which are 
isomorphic to a complex with flat, respectively projective, entries. Then  
$\Htp\FlatR$ and $\Htp\ProjR$ are triangulated subcategories of $\Htp\ModR$ (more generally, 
the same is true if we start with any additive subcategory of $\ModR$).

If $R$ and $S$ are rings, $X$ is a complex of $R$--$S$--bimodules and $V$ is a right $S$--module 
we denote by $\Hom_S(X,V)$ the complex of right $R$--modules:
\[\Hom_S(X,V)=\cdots\to\Hom_S(X^{n+1},V)\to\Hom_S(X^n,V)\to\cdots\]
where the differentials are the induced ones.

Let $\T$ be a triangulated category, and let $\A$ be an abelian category. 
We call {\em (co)homological} a (contravariant) functor 
$F:\T\to\A$ which sends triangles into long 
exact sequences. Denote by $\Ab$ the category of abelian groups. 
If $\T$ has coproducts (products) we say as in \cite{NR} that 
$\T$ (respectively, $\T\opp$) satisfies Brown representability, 
if every cohomological (homological) functor $F:\T\to\Ab$ which sends 
coproducts into products (preserves products) is representable. 

Now it is time to state our main result:

\begin{thm}\label{brt4kproj}
If $R$ is a ring then the category $\Htp\ProjR\opp$ satisfies Brown representability.
\end{thm}

\section{The proof}

The first ingredient in the proof of the main theorem of this paper is contained in 
\cite{MD}. Here we recall it shortly.
Fix $\T$ to be a triangulated category with products, and denote by $[1]$ its suspension functor. 
Recall that if \[X_1\leftarrow X_2\leftarrow X_3\leftarrow \cdots \] is an inverse tower 
(indexed over $\N$) of objects in $\T$, then its 
homotopy limit is defined (up to a non--canonical isomorphism) by the triangle 
\[\holim X_n\la\prod_{n\in\N^*}X_n\stackrel{1-shift}\la\prod_{n\in\N^*}X_n\to\holim X_n[1], \]
(see \cite[dual of Definition 1.6.4]{N}).

Consider a set of objects in $\T$ and denote it by $\CS$. We define $\Prod\CS$ to be the full
subcategory of $\T$ consisting of all direct factors of products of objects in $\CS$.
Next we define inductively $\Prd_1(\CS)=\Prod\CS$ and $\Prd_n(\CS)$ is
the full subcategory of $\T$ which consists of all objects $Y$
lying in a triangle $X\to Y\to Z\to X[1]$ with $X\in\Prd_1(\CS)$
and $Y\in\Prd_n(\CS)$. Clearly the construction leads to an ascending chain
$\Prd_1(\CS)\subseteq\Prd_2(\CS)\subseteq\cdots$. We suppose that $\CS$ is closed under
suspensions and desuspensions, hence the same is true for
$\Prd_n(\CS)$, by \cite[Remark 07]{NR}.  The same
\cite[Remark 07]{NR} says, in addition, that if $X\to Y\to Z\to X[1]$ is a
triangle with $X\in\Prd_n(\CS)$ and $\Prd_m(\CS)$ then
$Z\in\Prd_{n+m}(\CS)$. An object $X\in\T$ will be called
{\em$\CS$-cofiltered} if it may be written as a homotopy limit 
$X\cong\holim X_n$ of an inverse tower, with $X_1\in\Prd_1(\CS)$, and
$X_{n+1}$ lying in a triangle $P_{n}\to X_{n+1}\to X_n\to P_n[1],$
for some $P_n\in\Prd_1(\CS)$. Inductively we have
$X_n\in\Prd_n(\CS)$, for all $n\in\N^*$. 
The dual notion must be surely called {\em filtered}, and the terminology comes 
from the analogy with the filtered objects in an abelian category (see \cite[Definition 3.1.1]{GT}). 
Using further the same analogy, we say that $\T$ (respectively, $\T\opp$) is {\em deconstructible} if 
$\T$ has coproducts (products) and there is a set, which is not a proper class, of objects $\CS$ closed under suspensions and desuspensions 
such that every object $X\in\T$ is $\CS$--filtered (cofiltered). Note that we may define 
deconstructibility without closure under suspensions and desuspension, Indeed if every $X\in\T$ is 
$\CS$--(co)filtered, then it is also $\overline{\CS}$--(co)filtered, where  $\overline{\CS}$ is the closure of $\CS$ 
under suspensions and desuspensions.

\begin{lem}\label{decon}\cite[Theorem 8]{MD}
If $\T\opp$ is deconstructible, then $\T\opp$ satisfies Brown representability.
\end{lem}

In order to apply this result to the category $\Htp\ProjR$ we shall use the set cogenerators 
of this category constructed
in \cite{NKP}.  We consider the subcategories $\Htp\ProjR^\perp$ and 
$\left(\Htp\ProjR^\perp\right)^\perp$ of $\Htp\FlatR$, 
where the symbol $^\perp$ is always meant in $\Htp\FlatR$, that is 
\begin{align*}\Htp\ProjR^\perp=\{X\in\Htp\FlatR\mid&\Htp\FlatR(P,X)=0\\ &
 \hbox{ for all }P\in\Htp\ProjR\}
\end{align*}
and similar for double perpendicular. By formal non--sense we know that there is an equivalence of 
categories
\[\Htp\FlatR/\Htp\ProjR^\perp\stackrel{\sim}\la\left(\Htp\ProjR^\perp\right)^\perp\]
thus \cite[Remark 2.16]{NKF} implies the existence of an equivalences of categories
\[\Htp\ProjR\stackrel{\sim}\la\left(\Htp\ProjR^\perp\right)^\perp.\]
Note that the cogenerators constructed in \cite{NKP} lie naturally not in 
$\Htp\ProjR$ but in the equivalent category $\left(\Htp\ProjR^\perp\right)^\perp$. This is the 
reason for which we shall work with this last category which will be denoted by $\T$. 

\begin{lem}\label{thaspd}
 The category $\T$ is has products. 
\end{lem}

\begin{proof}
We know that $\T\sim\Htp\ProjR$ is well--generated (see \cite[Theorem 1.1]{NKF}), 
hence it satisfies Brown representability.
The existence of products is a well--known consequence of this fact: If $(X_i)_{i\in I}$ is a 
family of objects in $\T$, then the cohomological functor $\prod_{i\in I}\T(-,X_i)$ sends 
coproducts in products, therefore it is representable (by the product of the family $(X_i)_{i\in I}$).
\end{proof}

\begin{rem}\label{radj}
 Another proof of Lemma \ref{thaspd} goes as follows: The very definition of $\T$ implies 
 that it is closed under products in $\Htp\FlatR$, and it remains to show that this last category 
 has products. But this follows immediately form the fact that the inclusion functor 
 $\Htp\FlatR\to\Htp\ModR$ has a right adjoint (see \cite[Theorem 3.2]{NA}). Indeed for obtaining 
 the product in $\Htp\FlatR$ we 
 have only to apply this right adjoint to the product in $\Htp\ModR$. 
\end{rem}

Recall that a {\em test--complex} is defined in \cite[Definition 1.1]{NKP}
to be a bounded below complex $I$ of injective left modules satisfying the additional properties that 
$\Ho^n(I)=0$ for all but finitely many $n\in\Z$ and for those $n$ for which $\Ho^n(I)\neq0$, this 
module is isomorphic to subquotient of a finitely generated projective module. Here by 
$\Ho^n(I)$ we understand the $n$-th (left) $R$--module of cohomology of the complex $I$. Note that there 
is only a set of test complexes up to homotopy equivalence (see also \cite[Remark 1.2]{NKP}).   
Let $J:\Htp\ModR\to\Htp\FlatR$ the right adjoint of the inclusion functor 
(see also Remark \ref{radj}). Define $\CS$ to be the full subcategory of $\Htp\FlatR$ which contains 
exactly the objects of the form $J(\Hom_{\Z}(I,\Q/\Z))$ where $I$ runs over a set of representatives up to 
homotopy equivalence of all test--complexes. Observe that $\CS$ is a set and $\CS\subseteq\T$ by 
\cite[Lemmas 2.2 and 2.6]{NKP}.  We plan to complete our proof by showing that 
$\T$ is $\CS$--cofiltred. In order to do that, we shall use the (proof of) \cite[Theorem 4.7]{NKP}. 
Recall from \cite[Construction 4.3 and Theorem 5.9]{NKF} that the full subcategory
$\G$ of $\Htp\FlatR$ which contains a set of representatives (again up to homotopy equivalence) 
for those $G\in\Htp\FlatR$ which 
are bounded below complexes with finitely generated projective entries generates $\Htp\FlatR$ 
as a triangulated subcategory. We recall also that if $\C'\subseteq\C$ is a full subcategory of 
any category $\C$, then  
a map $Y\to Z$ in $\C$ with $Z\in\C'$ is called {\em a $\C'$-preenvelope of $Y$}, provided that 
every other map $Y\to Z'$ with $Z'\in\C'$ factors through $Y\to Z$. 

The next three lemmas are refinements of \cite[Lemmas 4.4, 4.5 and 4.6]{NKP}.

\begin{lem}\label{preen} Every complex $Y\in\Htp\FlatR$ has a $\Prod\CS$--preenvelope. 
\end{lem}

\begin{proof}
 The argument is standard: Let $Z=\prod_{S\in\CS,\alpha:Y\to S}S$ and $Y\to Z$ the unique map 
 making commutative the diagram:
 \[\diagram 
 Y\rrto\drto_{\alpha}&& Z\dlto^{p_{S,\alpha}}\\
 &S& 
 \enddiagram\] where $p_{S,\alpha}$ is the canonical projection for all $S\in\CS$ and all 
 $\alpha:Y\to S$. 
\end{proof}

\begin{rem}
In \cite[Lemma 4.4]{NKP} the map $Y\to Z$ from Lemma \ref{preen} is completed
to a triangle \[X\to Y\to Z\to X[1]\] and it is shown that the condition to be a 
$\Prod\CS$--preenvelope is equivalent to the fact that $X\to Y$ is a {\em tensor phantom map}, 
that is the induced map $X\otimes_RI\to Y\otimes_RI$ vanishes in cohomology for every test--complex 
$I$.
\end{rem}

\begin{lem}\label{z2}
 For every $Y\in\Htp\FlatR$ there is a triangle \[X\to Y\to Z\to X[1]\] such that $Z\in\Prd_2(\CS)$ 
 and the induced sequence 
 \[0\to\Htp\FlatR(G,Y)\to\Htp\FlatR(G,Z)\to\Htp\FlatR(G,X[1])\to0\]
is exact for all $G\in\G$. 
\end{lem}

\begin{proof} Use twice Lemma \ref{preen}: First consider a $\Prod\CS$--preenvelope $Y\to Z'$,
and complete it to a triangle $X'\to Y\to Z'\to X'[1]$. Let $X'\to Z''$ be again a 
$\Prod\CS$--preenvelope which is completed to a triangle $X\to X'\to Z''\to X[1]$. The octraedral 
axiom allows us to construct the commutative diagram whose rows and columns are triangles:
\[\diagram
             &Z'[-1]\rdouble\dto&Z'[-1]\dto &            \\
X\rto\ddouble&X'\rto\dto        &Z''\rto\dto&X[1]\ddouble\\
X\rto        &Y\rto\dto         &Z\rto\dto  &X[1]        \\
             &Z'\rdouble        &Z'         &
\enddiagram.\]
 Now the third row is the desired triangle, and the triangle 
 in the third column assures us that $Z\in\Prd_2(\CS)$. The rest of the conclusion follows by 
 \cite[Lemma 4.5]{NKP}.
\end{proof}

\begin{lem}\label{z3}
 Every map $Y\to Z$ in $\Htp\FlatR$ with $Z\in\Prd_n(\CS)$ factors as $Y\to Z'\to Z$, where 
 $Z'\in\Prd_{n+2}(\CS)$ and the induced maps 
 \[\Htp\FlatR(G,Y)\to\Htp\FlatR(G,Z)\]  \[\Htp\FlatR(G,Z')\to\Htp\FlatR(G,Z)\]
 have the same image, for all $G\in\G$. 
\end{lem}

\begin{proof}
 Complete $Y\to Z$ to a triangle $Y\to Z\to X\to Y[1]$ and let $X\to Z''$ as in Lemma \ref{z2}.
 Complete the composed map $Z\to X\to Z''$ to a triangle 
 \[Z'\to Z\to Z''\to Z'[1].\] It is clear that $Z'\in\Prd_{n+2}(\CS)$ and the rest of the proof is 
 the same as for \cite[Lemma 4.6]{NKP}.
\end{proof}

\begin{pf} {\it Theorem \ref{brt4kproj}}. Fix $Y\in\Htp\FlatR$. Construct $Y\to Z_1$, with 
$Z_1\in\Prd_2(\CS)$ as in Lemma \ref{z2}. Inductively the map $Y\to Z_n$, with $Z_n\in\Prd_{2n}(\CS)$, 
$n\in\N^*$, factors as $Y\to Z_{n+1}\to Z_n$, with $Z_{n+1}\in\Prd_{2(n+1)}(\CS)$, according to  
Lemma \ref{z3}. The argument used in the proof of \cite[Theorem 4.7]{NKP} leads to a triangle 
\[X\to Y\to Z\to X[1]\] such that $X\in\Htp\ProjR^\perp$ and $Z=\holim Z_n$. 
In particular this shows that if $Y\in\T$ then $Y\cong Z=\holim Z_n$, hence $\T$ is $\CS$--cofiltered.
\qed 
\end{pf}

We end this note by pointing out that the existence of the left adjoint of the inclusion functor 
$\T\to\Htp\FlatR$ is a consequence of Theorem \ref{brt4kproj}. By now there are several proof 
of this fact (see \cite{NKP}), but the new one is deduced more conceptually 
from Brown representability. 

\begin{cor}\label{aft} Let $\U$ be a triangulated category (with small hom--sets). 
 If $F:\Htp\ProjR\to\U$ is a product preserving functor that $F$ has a left adjoint. 
In particular the inclusion functor $\T\to\Htp\FlatR$ has a left adjoint.
 \end{cor}

\end{document}